\documentclass[10pt]{amsart}
\usepackage{amsmath,amssymb,amsthm,amsfonts,amscd,mathrsfs,mathtools}
\usepackage[utf8]{inputenc}
\usepackage{tikz-cd}
\usepackage{enumitem}
\usepackage[all]{xy}
\usepackage[hyphens]{url}
\usepackage{caption}
\usepackage{hyperref}
\usepackage{verbatim}

\setlist[itemize]{%
labelsep=8pt,%
labelindent=0.5\parindent,%
itemindent=0pt,%
leftmargin=*,%
listparindent=-\leftmargin%
}

\theoremstyle{plain}
    \newtheorem{thm}{Theorem}[section]
    \newtheorem{lem}[thm]   {Lemma}
    \newtheorem{cor}[thm]   {Corollary}

\theoremstyle{definition}
    \newtheorem{defn}[thm]  {Definition}

    \newcommand{\Z}{\mathbb{Z}}

\newcommand{\C}{\mathbb{C}}

\newcommand{\fe}{\mathfrak{e}}

\newcommand{\F}{\mathbb{K}}
\newcommand{\fF}{\mathfrak{F}}
\newcommand{\M}{\mathcal{M}}
\newcommand{\tR}{\tilde R}

\renewcommand{\tt}{\tilde t}
\newcommand{\tF}{\tilde \F}

\newcommand{\tor}{\tau}

\newcommand{\set}[1]{\left\{#1\right\}}
\newcommand{\pa}[1]{\left(#1\right)}
\newcommand{\abs}[1]{\left|#1\right|}

\newcommand{\ab}{\mathrm{ab}}
\newcommand{\cPn}{[P_n,P_n]}
\newcommand{\Fnlog}{F_n^{\log}}
\newcommand{\tFnlog}{\tilde F_n^{\log}}

\newcommand{\Ch}{\mathfrak{Ch}}
\newcommand{\Chb}{\Ch_{\bullet}}
\newcommand{\tChb}{\tilde\Chb}

\newcommand{\Chlog}{\Ch^{\log}}
\newcommand{\Chlogb}{\Chlog_{\bullet}}
\newcommand{\tChlog}{\tilde\Ch^{\log}}
\newcommand{\tChlogb}{\tChlog_{\bullet}}
\newcommand{\MtChlogb}{\M\tChlogb}

\newcommand{\Sal}{Sal}
\newcommand{\Sallog}{\Sal^{\log}}

\renewcommand{\i}{\iota}
\renewcommand{\phi}{\varphi}

\DeclareMathOperator{\cd}{cd}

\author{Andrea Bianchi}
\title[Homology of commutators of pure braids]{On the homology of the commutator subgroup of the pure braid group}
\date{\today}

\address{Mathematics Institute,
University of Bonn,
Endenicher Allee 60, Bonn,
Germany
}

\email{bianchi@math.uni-bonn.de}


\thanks{This work was supported by the Deutsche Forschungsgemeinschaft (DFG, German Research
Foundation) under Germany’s Excellence Strategy (EXC-2047/1, 390685813).}

\begin{document}

\begin{abstract}
 We study the
 homology of $\cPn$, the commutator subgroup of the pure braid group on $n$ strands, and
 show that $H_l(\cPn)$ contains a free abelian group of infinite rank for all
 $1\leq l\leq n-2$. As a consequence we determine the cohomological dimension of $\cPn$:
 for $n\geq 2$ we have $\cd(\cPn)=n-2$. 
\end{abstract}
\maketitle

\section{Introduction}
Let $n\geq 2$ and denote by $F_n$ the \emph{ordered configuration space} of $n$ points in the complex plane:
 \[
  F_n=\set{(z_1,\dots,z_n)\in\C^n \,|\, z_i\neq z_j\;\forall i\neq j}.
 \]
 The \emph{pure braid group} on $n$ strands is defined as $P_n=\pi_1(F_n)$.

In \cite{BianchiRecio} David Recio-Mitter and the author posed the question of determining the cohomological
dimension of $\cPn$, the commutator subgroup of the pure braid group, and conjectured that, for $n\geq 2$,
\[
 \cd(\cPn)=n-2.
\]

In this work we prove this conjecture by computing a large part of the homology of $\cPn$; in particular
we prove that $H_*(\cPn)$ contains a free abelian group of infinite rank in all degrees $1\leq *\leq n-2$ (see Theorem \ref{thm:HtensortF},
Corollary \ref{cor:cdcPn} and Theorem \ref{thm:middle}).

To the best of the author's knowledge there is no result in the literature concerning the homology of $\cPn$ for large values of $n$;
on the contrary the homology of the commutator subgroup of Artin's \emph{full}
braid group \cite{Artin} has been extensively studied \cite{Frenkel88, Markaryan96, DcPS01, Callegaro06}, as well as the homology
of Milnor fibers of discriminant fibrations associated with other hyperplane arrangements in $\C^n$ \cite{DcPSS99, Denham02, CaMoSal08, Settepanella09}.

Our strategy is the following. We consider the Salvetti complex $\Sal_n$ associated with the $n$-th braid arrangement:
the cell complex $\Sal_n$ is a classifying space for $P_n$, and it has a
covering $\Sallog_n$ which is a classifying space for $\cPn$. The group $P_n^{\ab}\simeq \Z^{\binom{n}{2}}$
acts on $\Sallog_n$ by deck transformations, and the action is cellular: hence the associated cellular chain complex
$\Chlogb$ is a chain complex of modules over the commutative ring $\Z[P_n^{\ab}]$, and consequently
the homology $H_*\pa{\cPn}$ is also a $\Z[P_n^{\ab}]$-module.

We replace $\Chlogb$ with a homotopy equivalent subcomplex $\tChlogb$; the chain complex $\tChlogb$ is only
invariant for the action of a certain subgroup $\Z^{\binom{n}{2}-1}\subset P_n^{\ab}$, and we restrict this action
also in homology, i.e. we consider $H_*\pa{\cPn}$ as a module over 
the commutative ring $\Z\left[\Z^{{\binom{n}{2}}-1}\right]$.

We define a filtration on $\tChlogb$; the associated Leray spectral sequence, after localisation to the
quotient field of $\Z\left[\Z^{{\binom{n}{2}}-1}\right]$, collapses on its first page: more precisely we have $E^1_{p,q}=0$
for all $(p,q)\neq (n-2,0)$. This proves the statement for $H_{n-2}(\cPn)$ (see Theorem \ref{thm:HtensortF}).

To prove the statement in lower degrees we consider the interaction between commutator subgroups of different pure
braid groups (see Theorem \ref{thm:middle}).

\section{Preliminaries}
We recall some classical constructions and results about configuration spaces and pure braid groups.

For all $1\leq i\leq n+1$ there is a map $\phi_i\colon F_{n+1}\to F_n$, which
forgets the $i-$th point of each configuration. This is a fiber bundle with fiber the punctured plane $\C\setminus\set{n\mbox{ points}}$,
called the Fadell-Neuwirth fibration (see \cite{FadellNeuwirth}):
\[
\begin{tikzcd}
  \C\setminus\set{n\mbox{ points}}\ar[r] & F_{n+1}\ar[r,"\phi_i"] & F_n.
\end{tikzcd}
\]
The space $\C\setminus\set{n\mbox{ points}}$ is a classifying space for the free group on $n$ generators $\Z^{*n}$,
in particular it is an aspherical space. An induction argument shows that $F_n$ is also aspherical, and
therefore $F_n$ is a classifying space for its fundamental group $P_n$.
We obtain a short exact sequence
\[
 1\to \Z^{*n}\to P_{n+1}\to P_n\to 1.
\]
\begin{defn}
\label{defn:Pnab}
For all $1\leq i<j\leq n$ there is a forgetful map $\psi_{ij}\colon F_n\to F_2$, which forgets
all points of a configuration except the $i$-th and the $j$-th. This map of spaces induces a map, that we still
call $\psi_{ij}$, on fundamental groups:
\[
 \psi_{ij}\colon P_n\to P_2\simeq\Z.
\]
The collection of all these maps gives a homomorphism of groups $\psi\colon P_n\to\Z^{\binom{n}{2}}$.
\end{defn}
A classical result by Arnold \cite{Arnold:coloredbraid} states that $\psi$ is
the abelianisation homomorphism, i.e. $P_n^{\ab}\simeq\Z^{\binom{n}{2}}$ along the map induced by $\psi$.
In this article we focus on the group $\cPn=\ker\psi$, the commutator subgroup of the pure braid group.

\section{Two classifying spaces for \texorpdfstring{$\cPn$}{[Pn,Pn]}}
We introduce two convenient models for the classifying space of $\cPn$.

\begin{defn}
We define the space $\Fnlog$. A point in $\Fnlog$ is determined by a configuration
$(z_1,\dots,z_n)\in F_n$ together with a choice $w_{ij}\in\C$ of a logarithm of $(z_j-z_i)$, for all $i<j$:
\[
 \Fnlog=\set{\pa{\pa{z_i}_{1\leq i\leq n},\pa{w_{ij}}_{1\leq i<j\leq n}}\,|\,z_j-z_i=e^{w_{ij}}\quad\forall 1\leq i<j\leq n}.
\]
This space has a topology as subspace of $\C^n\times \C^{\binom{n}{2}}$.
\end{defn}
There is a covering map $p\colon\Fnlog\to F_n$, which forgets the numbers $w_{ij}$. The fiber is isomorphic to
$\Z^{\binom{n}{2}}$: to see this fix a point $\pa{(z_i),(\bar w_{ij})}$ lying over some point $(z_i)\in F_n$.
Let $\pa{(z_i),(w_{ij})}$ be any
other point lying over $(z_i)$: then there are integers $(k_{ij})_{1\leq i<j\leq n}$ such that
$w_{ij}-\bar w_{ij}=2\pi\sqrt{-1}k_{ij}$ for all $1\leq i<j\leq n$. Viceversa
given integers $(k_{ij})_{1\leq i<j\leq n}$ one can define a point $\pa{(z_i),(w_{ij})}$ in the fiber of $(z_i)$ by setting
$w_{ij}=\bar w_{ij}+2\pi\sqrt{-1}k_{ij}$ for all $1\leq i<j\leq n$.

The last construction gives a free action of $\Z^{\binom{n}{2}}$ on $\Fnlog$; this is an action by
deck transformations of $p$ and is transitive on fibers of $p$: therefore $\Z^{\binom{n}{2}}$ 
is the whole group of deck transformations of $p$ and there is a short
exact sequence
\[
 1\to \pi_1(\Fnlog)\to \pi_1(F_n)\to \Z^{\binom{n}{2}}\to 1.
\]
We can then conclude that $\cPn\subseteq \pi_1(\Fnlog)$, because $\cPn$ is contained in the kernel of any map
from $P_n$ to an abelian group.

On the other hand the maps
$\psi_{ij}\colon F_n\to F_2$ lift to maps $\psi_{ij}^{\log}\colon \Fnlog\to F_2^{\log}$: the map $\psi_{ij}^{\log}$ is defined
by forgetting all data except $z_i,z_j$ and $w_{ij}$.

The space $F_2^{\log}$ is contractible: this is a particular case
of Lemma \ref{lem:tFnlog&Fnlog}, and can be checked also directly.
Therefore $\pi_1(\Fnlog)$ is a subgroup
of $P_n$ contained in the kernel of all maps $\psi_{ij}$, i.e. $\pi_1(\Fnlog)\subseteq\cPn$.
We obtain the following lemma.
\begin{lem}
 \label{lem:Fnlog=BcPn}
 The space $\Fnlog$ is a classifying space for the group $\cPn$.
\end{lem}
The action of $P_n^{\ab}$ on $\Fnlog$
induces an action of the ring $\Z[P_n^{\ab}]$ on $H_*(\Fnlog)$, so our first attempt is to study
$H_*(\Fnlog)=H_*(\cPn)$ as a module over this ring.

\begin{defn}
Let $R(n)=\Z[P_n^{\ab}]$ be the ring of Laurent
polynomials in $\binom{n}{2}$ variables $\Z\left[t_{ij}^{\pm 1}|1\leq i<j\leq n\right]$. The variable $t_{ij}$
corresponds to the generator of $P_n^{\ab}\simeq \Z^{\binom{n}{2}}$ which is \emph{dual} to the map
$\psi_{ij}\colon P_n\to P_2$, i.e. for all $i<j$ and $k<l$ we have $\psi_{ij}(t_{kl})=\delta_{ik}\delta_{jl}$.

The ring $R(n)$ is a domain and we call $\F(n)$
its quotient field.
\end{defn}
The following lemma tells us that $H_*(\Fnlog)$ cannot be too \emph{large}.
\begin{lem}
\label{lem:H*istauntorsion}
\[
 H_*(\Fnlog)\otimes_{R(n)} \F(n)=0.
\]
\end{lem}
\begin{proof}
 Consider the following homotopy $\mathrm{H}\colon\Fnlog\times[0,2\pi]\to\Fnlog$ of the space $\Fnlog$ into itself. 
 At time $0$, the map $\mathrm{H}(\cdot;0)$ is the identity of $\Fnlog$; at time $\theta$ we rotate each configuration
 by an angle $\theta$ counterclockwise, adjusting logarithms:
 \[
 \mathrm{H}\pa{((z_i),(w_{ij}));\theta}=\pa{ (e^{\theta\sqrt{-1}}z_i),(w_{ij}+\theta\sqrt{-1})}.
 \]
 At time $2\pi$, the map $\mathrm{H}(\cdot;2\pi)$ preserves all $z_i$'s
 and shifts all $w_{ij}$'s by $2\pi\sqrt{-1}$: this last map is precisely the map
 \[
 \prod_{1\leq i<j\leq n}t_{ij}\colon\Fnlog\to\Fnlog,
 \]
 i.e. the product of all deck transformations $t_{ij}\colon\Fnlog\to\Fnlog$.
 
 Since $\prod_{1\leq i<j\leq n}t_{ij}$ is homotopic to the identity of
 $\Fnlog$, it induces the identity map on $H_*\pa{\Fnlog}$.
 
 Hence $H_*(\Fnlog)$, as a $R(n)$-module, is $\left[(\prod_{\leq i<j\leq n}t_{ij})-1\right]$-torsion, in particular
 its $\F(n)$-localisation vanishes.  
\end{proof}

The proof of the previous lemma tells us that the variable $t_{12}$ acts on $H_*(\cPn)$ as the product
$\prod_{(i<j)\neq (1,2)}t_{ij}^{-1}$; therefore it seems convenient to replace $R(n)$ with a \emph{smaller} ring, containing
one variable less.

\begin{defn}
\label{defn:tR}
 We call
 \[
 \tR(n)=\Z\left[\tilde t_{ij}^{\pm 1} \,|\,1\leq i<j\leq n\, , \,(i,j)\neq(1,2)\right]
 \]
 the ring of Laurent
 polynomials in ${\binom{n}{2}}-1$ variables. $\tR(n)$ is naturally a subring of $R(n)$ by identifying each $\tt_{ij}$
 with the corresponding $t_{ij}$, and therefore each $R(n)$-module is also a $\tR(n)$-module.
 
 We can also identify $\tR(n)$ as the quotient of $R(n)$ by the ideal generated by the element $\pa{\prod_{1\leq i<j\leq n}t_{ij}}-1$.
 The composition of maps of rings $\tR(n)\subset R(n)\to\tR(n)$ is the identity of $\tR(n)$.
 
 The ring $\tR(n)$ is a
 domain, and we call $\tF(n)$
 its quotient field.
\end{defn}

We want now to study $H_*(\cPn)$ as a $\tR(n)$-module. We introduce our second model of a classifying space for $\cPn$.
\begin{defn}
 The space $\tFnlog$ is defined as the subspace of $\Fnlog$ of configurations $\pa{(z_i),(w_{ij})}$
 such that $z_2=1$, $z_1=0$ and $w_{12}=0$.
\end{defn}
The space $\tFnlog$ is not invariant under the action of the whole group $P_n^{\ab}$ on $\Fnlog$: the action of 
$t_{12}$ consists in shifting $w_{12}$ by $2\pi\sqrt{-1}$, and this is not allowed inside $\tFnlog$.
The other generators $\tt_{ij}$ of $P_n^{\ab}$ preserve $\tFnlog$; we conclude that $H_*(\tFnlog)$ has a
natural structure of $\tR(n)-$module, and the inclusion map $\tFnlog\subset\Fnlog$ induces a map
of $\tR(n)$-modules in homology.

\begin{lem}
\label{lem:tFnlog&Fnlog}
 $\tFnlog$ is a deformation retract of $\Fnlog$, and therefore it is also a classifying space for $\cPn$.
\end{lem}
\begin{proof}
 We define a homotopy $\mathrm{H}\colon \Fnlog\times[0,1]\to\Fnlog$ starting with the identity of $\Fnlog$ and ending
 with a retraction onto $\tFnlog$; the space $\tFnlog$ will be fixed pointwise throughout the homotopy.
 
 Let $\pa{(z_i),(w_{ij})}\in\Fnlog$. Then
 \[
  \mathrm{H}(\pa{(z_i),(w_{ij})};t)=\pa{(e^{-tw_{12}}\cdot (z_i-tz_1)),(w_{ij}-tw_{12})}.
 \]
\end{proof}

\section{Chain complexes}
In this section let $n\geq 2$ be fixed. Our next aim is to describe explicitly a chain complex that
computes the homology of $\Fnlog$. We first recall the classical
chain complex computing the homology of $F_n$: it can be seen both as the dual of the reduced cochain complex
of the one-point-compactification of $F_n$, in the spirit of Fuchs \cite{Fuchs:CohomBraidModtwo},
or as chain complex associated with the Salvetti complex \cite{Salvetti} of the $n$-th braid arrangement.
\begin{defn}
 \label{defn:Ch}
An \emph{ordered partition} of $\set{1,\dots,n}$ of degree $1\leq k\leq n$ is a partition of $\set{1,\dots, n}$ into $n-k$
 non-empty subsets $(\pi_1,\dots,\pi_{n-k})$, where each \emph{piece} $\pi_r$ is endowed with a total order.
 
 For $a,b\in\pi_r$ we write $a\prec b$ if $a$ \emph{precedes} $b$ in the order associated with $\pi_r$, and
 we keep writing $a<b$ if $a$ is smaller than $b$ as natural numbers.

 We define the chain complex $\Chb=\Chb(n)$. Let $\Ch_k$ be the free abelian group with one generator (also called \emph{cell})
 for each ordered partition of $\set{1,\dots,n}$ of degree $k$.
 
In order to describe the boundary maps of $\Chb$ it is enough, for any two ordered partitions $(\pi_r)_{1\leq r\leq n-k}$
and $(\pi'_r)_{1\leq r\leq n-k+1}$
of degree $k$ and $k-1$ respectively, to give a formula for the \emph{boundary
index} $[\partial(\pi_r)\colon(\pi'_r)]$,
i.e. the coefficient of $(\pi'_r)$ in $\partial(\pi_r)$. There are two possibilities:
\begin{itemize}
 \item $(\pi'_r)$ is obtained from $(\pi_r)$ by
 \begin{itemize}
  \item  splitting some piece $\pi_l$ into two pieces $\pi'_l$ and $\pi'_{l+1}$,
 each having as total order the restriction of the total order on $\pi_l$;
 \item setting $\pi'_r=\pi_r$ for $r<l$ and $\pi'_r=\pi_{r-1}$ for $r>l+1$, with the same total orders.
  \end{itemize}
  Then
  \[
   [\partial(\pi_r)\colon(\pi'_r)]=(-1)^{l+sgn(\pi_l;\pi'_l)}=\pm 1,
  \]
 where, for an ordered set $(A,\prec)$ and a subset $B$, we define $sgn(A,B)$ as the parity of the number
 of couples $(a,b)$ of elements of $A$ with $b\prec a$, $b\in B$ and $a\not\in B$.
 \item $(\pi'_r)$ is not obtained from $(\pi_r)$ as before. Then $[\partial(\pi_r)\colon(\pi'_r)]=0$.
  \end{itemize}
\end{defn}
The chain complex $\Chb$ is the cellular chain complex of the Salvetti complex $\Sal_n$:
it is a finite cell complex contained in $F_n$, onto which $F_n$ deformation retracts \cite{Salvetti}.

Alternatively, in the spirit of Fuchs \cite{Fuchs:CohomBraidModtwo}, one can consider the following stratification of $F_n$.
For every ordered partition $(\pi_r)_{1\leq r\leq n-k}$ of some degree $k$, we consider the subspace $e(\pi_r)\subset F_n$
consisting of all configurations $(z_1,\dots, z_n)$ satisfying the following properties:
\begin{itemize}
 \item there are exactly $n-k$ vertical lines in $\C$ passing through some of the $n$ points;
 \item for $1\leq r\leq n-k$, the $r$-th vertical line from left contains precisely the points $z_i$ with
 $i\in\pi_r$, and these points are assembled from the top to the bottom according to the total order $\prec$.
\end{itemize}
In particular for configurations $(z_i)\in e(\pi_r)$ the following properites hold:
\begin{itemize}
 \item for all $i\neq j$, if $i\in\pi_l$ and $j\in\pi_{l'}$ with
$l<l'$, the point $z_i$ lies \emph{on left} of the point $z_j$, i.e. $\Re(z_i)<\Re(z_j)$;
\item for all $i\neq j$, if both $i$ and $j$ belong to the same
piece $\pi_l$ and if $i\prec j$, then $z_i$ lies \emph{above} $z_j$, or equivalently $\Im(z_i)>\Im(z_j)$.
\end{itemize}

The one-point compactification $F_n^+$ of $F_n$ has a CW structure given by the subspaces 
$e(\pi_r)$ together with the point at infinity $\infty$. The associated reduced cellular cochain complex is precisely the one described in
Definition \ref{defn:Ch}. Note that each cell $e(\pi_r)$ is modeled on the interior of a product of simplices
\[
 \Delta^{n-k}\times\Delta^{\abs{\pi_1}}\times\cdots\times\Delta^{\abs{\pi_{n-k}}}.
\]
The local coordinates
are the \emph{horizontal} positions of the $n-k$ vertical lines and the \emph{vertical} positions of the points $z_i$
on these lines. We regard $e(\pi_r)$ as a manifold of dimension $2n-k$; an orientation can be given by declaring a total order
on the simplicial local coordinates, and we choose the lexicographic order associated with the product structure written above.

With this convention, the boundary index $[\partial e(\pi'_r)\colon e(\pi_r)]$ in the reduced cellular chain complex of $F_n^+$ equals
the formula for $[\partial (\pi_r)\colon (\pi'_r)]$ in Definition \ref{defn:Ch}.

The space $F_n$ is a $2n$-dimensional manifold and its stratification by the subspaces $e(\pi_r)$ gives rise to a Poincar\'e-dual cell complex, which is
exactly the Salvetti complex $\Sal_n$.

The space $\Sal_n$ has a covering $\Sallog_n$ corresponding to the subgroup $\cPn$ of its fundamental group
$P_n$, and we can lift to $\Sallog_n$ the cell complex structure on $\Sal_n$. The group of deck transformations
$P_n^{\ab}$ acts freely on the cells of $\Sallog_n$; the associated chain complex is
a chain complex of finitely generated, free $R(n)$-modules.
\begin{defn}
\label{defn:Chlog}
 We define a chain complex $\Chlogb$. Let $\Chlog_k$ be the free abelian group with one generator
 (called \emph{cell})
 for each choice of the following set of data:
 \begin{itemize}
  \item an ordered partition $(\pi_r)_{1\leq r\leq n-k}$ of $\set{1,\dots,n}$ of degree $k$;
  \item integers $W_{ij}\in\Z$ for all $1\leq i<j\leq n$.
 \end{itemize}
The boundary map has a similar formula as in Definition \ref{defn:Ch}. Consider cells $(\pi_r,W_{ij})_{1\leq r\leq n-k}$
and $(\pi'_r,W'_{ij})_{1\leq r\leq n-k+1}$
in degrees $k$ and $k-1$ respectively.
\begin{itemize}
 \item Suppose that the ordered partition $(\pi'_r)$ is obtained from $(\pi_r)$ as
 in the first case in Definition \ref{defn:Ch},
 splitting some $\pi_l$ into $\pi'_l$ and $\pi'_{l+1}$. Suppose that for all $i<j$ 
 satisfying
 \begin{itemize}
  \item  $i,j\in\pi_l$;
  \item $i\prec j$ in $\pi_l$;
  \item $i\in\pi'_l$ and $j\in\pi'_{l+1}$
 \end{itemize}
 we have $W'_{ij}=W_{ij}+1$. Finally, suppose that for all \emph{other} couples of indices $i<j$
 we have $W'_{ij}=W_{ij}$.
 Then
  \[
   [\partial(\pi_r,W_{ij})\colon(\pi'_r,W'_{ij})]=(-1)^{l+sgn(\pi_l;\pi'_l)}=\pm 1.
  \]
 \item If $(\pi'_r,W'_{ij})$ cannot be obtained from $(\pi_r,W_{ij})$ as before, then
the boundary index is zero.
\end{itemize}
\end{defn}
Similarly as before, we can stratify $\Fnlog$ as follows: for all $(\pi_r,W_{ij})_{1\leq r\leq n-k}$ as in Definition
\ref{defn:Chlog}, consider the subspace $e(\pi_r,W_{ij})$ of $\Fnlog$ determined by the following properties:
\begin{itemize}
 \item $e(\pi_r,W_{ij})$ is a connected component of $p^{-1}(e(\pi_r))$, where $p\colon\Fnlog\to F_n$ is the usual covering map;
 \item for all $i<j$, there exists a configuration $((z_i),(w_{ij}))\in e(\pi_r,W_{ij})$, depending on $i$ and $j$,
 such that one of the following four situations occurs, depending on the position of $i$ and $j$ in the ordered partition $(\pi_r)$:
 \begin{itemize}
 \item $z_j=z_i+1$ and $w_{ij}=2\pi\sqrt{-1}\pa{W_{ij}}$, assuming $i\in\pi_l$ and $j\in\pi_{l'}$ for some $l<l'$;
 \item $z_j=z_i+\sqrt{-1}$ and $w_{ij}=2\pi\sqrt{-1}\pa{W_{ij}+\frac 14}$, assuming $i,j\in\pi_l$ for some $l$, and $j\prec i$;
 \item $z_j=z_i-1$ and $w_{ij}=2\pi\sqrt{-1}\pa{W_{ij}+\frac 12}$, assuming $i\in\pi_l$ and $j\in\pi_{l'}$ for some $l>l'$;
 \item $z_j=z_i-\sqrt{-1}$ and $w_{ij}=2\pi\sqrt{-1}\pa{W_{ij}+\frac 34}$, assuming $i,j\in\pi_l$ for some $l$, and $i\prec j$.
 \end{itemize}
\end{itemize}
This stratification is the pull-back along $p$ of the stratification on $F_n$.
 We can add a point $\infty$ to $\Fnlog$ and obtain a space $\pa{\Fnlog}^+$ with a CW structure with
the cells $e(\pi_r,W_{ij})$ together with the point $\infty$.

The space $\pa{\Fnlog}^+$ is not the one-point compactification of $\Fnlog$,
but it is universal among topological spaces satisfying the following properties:
\begin{itemize}
 \item $\pa{\Fnlog}^+$ is obtained from $\Fnlog$ by adding one point $\infty$;
 \item for every $X\subset \Fnlog$ meeting finitely many strata $e((\pi_r),(W_{ij}))$, the closure of $X$ in
 $\pa{\Fnlog}^+$ is compact.
\end{itemize}

The genuine one-point compactification of $\pa{\Fnlog}$ would have a coarser topology than $\pa{\Fnlog}^+$, and in particular
it would not have the topology of a CW complex.

The chain complex $\Chlog$ coincides with the complex of reduced, compactly supported cochains of $\pa{\Fnlog}^+$;
the formulas for the indices are the same because we lift the canonical orientations of cells $e(\pi_r)\subset F_n$
to their preimages along $p$. The manifold $\Fnlog$ is stratified by the spaces $e((\pi_r),W_{ij})$ and there is a Poincar\'e
dual cell complex, which is
precisely the covering $\Sallog_n$ of the Salvetti complex $\Sal_n$.

Putting together all $\Z$-summands generated by cells $(\pi_r,W_{ij})\in\Chlog$ for fixed $(\pi_r)$ and varying $(W_{ij})$
we obtain one $R(n)$-summand of $\Chlogb$: the action of $P^{\ab}_n$ on this summand is analogous to the one
discussed for the space $\Fnlog$ (see the discussion preceding Lemma \ref{lem:Fnlog=BcPn}): multiplication times
$t_{kl}$ consists in shifting the number $W_{kl}$ by 1, while keeping the other numbers $W_{ij}$ as well as the 
ordered partition $(\pi_r)$.

We note that $\Chlogb$ is a chain complex of finitely generated, free
$R(n)$-modules; a $R(n)$-basis is given by those elements $(\pi_r,W_{ij})\in\Chlogb$ with $W_{ij}=0$ for all $i,j$;
we call these basis elements $(\pi_r,0)\in\Chlogb$ to distinguish them from the elements $(\pi_r)\in\Chb$
generating $\Chb$ over $\Z$.

The differentials of $\Chlogb$ with respect to the basis of the elements $(\pi_r,0)$
are expressed in a similar way as in Definition
\ref{defn:Chlog}, but boundary indices are no longer always equal to $0$ or $\pm 1$, rather
they can take the form of a product of some variables $t_{ij}^{\pm 1}$, with a sign $\pm 1$
determined in the same way as in Definition \ref{defn:Chlog}.
It is however still true that all boundary indices of $\Chlogb$ are either $0$ or invertible elements of $R(n)$.

There is a natural map $\Chlogb\to\Chb$ of chain complexes of abelian groups, mapping the generator
$(\pi_r,W_{ij})$ to the generator $(\pi_r)$: this map is induced by the covering map $\Sallog_n\to\Sal_n$, which by
construction is a cellular map.

\begin{defn}
\label{defn:tChnlog}
The chain complex $\Chlogb$ contains a subcomplex $\tChlogb$ of free abelian groups generated by cells $(\pi_r,W_{ij})$
such that:
\begin{itemize}
 \item there are indices $l<l'$ with $1\in\pi_l$ and $2\in\pi_{l'}$;
 \item $W_{12}=0$.
\end{itemize}
\end{defn}
Note that $\tChlogb$ is a subcomplex of abelian groups, and in particular is closed along boundary maps:
if $(\pi_r,W_{ij})$ is a generator of $\tChlogb$, then $1$ and $2$ already belong to different pieces of
the partition $(\pi_r)$, so that $W_{12}$ cannot change along boundaries, according to Definition \ref{defn:Chlog}.
The degrees of cells in $\tChlogb$ range from $0$ to $n-2$, because there are always at least two pieces in the partition.

\begin{lem}
 \label{lem:tChnlog=ChtFnlog}
The chain complex $\tChlogb$ computes the homology of $\tFnlog$.
\end{lem}
\begin{proof}
The space $\tFnlog$ can be also defined as follows. Let $F_{n-2}(\C\setminus\set{0,1})$ be the
subspace of $F_n$ of configurations $(z_1,\dots, z_n)$ with $z_1=0$ and $z_2=1$. The space
$F_{n-2}(\C\setminus\set{0,1})$
is the ordered configuration space of $n-2$
points in the $2$-punctured plane, so it is the fiber over $(z_1=0,z_2=1)$ of the
bundle map $\psi_{12}\colon F_n\to F_2$ forgetting
all points but the first two (see Definition \ref{defn:Pnab}).

The space $F_{n-2}(\C\setminus\set{0,1})$ is aspherical, and its fundamental group is
the kernel of the map induced by $\psi_{12}$ on fundamental groups; moreover
$H_1(F_{n-2}(\C\setminus\set{0,1}))\simeq\Z^{{\binom{n}{2}}-1}$, where the isomorphism is
exhibited by the collection of maps $\psi_{ij}$ with $(i,j)\neq(1,2)$.

{The commutator subgroup of $\pi_1\pa{F_{n-2}(\C\setminus\set{0,1})}$ can be identified with\unskip\parfillskip 0pt\par}
\noindent $\cPn$, and $\tFnlog$ is the covering of $F_{n-2}(\C\setminus\set{0,1})$ corresponding
to this group.

The space $F_{n-2}(\C\setminus\set{0,1})$ is the complement in $\C^{n-2}$
of a hyperplane arrangement: using $z_3,\dots,z_n$ as coordinates of $\C^{n-2}$
we are considering the following hyperplanes with real equations
\begin{itemize}
 \item $z_i=0$, for $3\leq i\leq n$;
 \item $z_i=1$, for $3\leq i\leq n$;
 \item $z_i=z_j$, for $3\leq i<j\leq n$.
\end{itemize}
Hence also $F_{n-2}(\C\setminus\set{0,1})$ deformation retracts onto a Salvetti complex, that we call $\tilde\Sal_n\subset F_{n-2}(\C\setminus\set{0,1})$.
Using the definition of the Salvetti complex \cite{Salvetti}
it is straightforward to check that the cellular chain complex $\tChb$ of $\tilde\Sal_n$ is isomorphic to the subcomplex of $\Chb$ generated
by cells $(\pi_r)$ satisfying the first condition of Definition \ref{defn:tChnlog}.

Another possibility is the following.
For every ordered partition $(\pi_r)$ satisfying the first condition of Definition \ref{defn:tChnlog}, we can consider the subspace
\[
\tilde{e}(\pi_r)\subset F_{n-2}(\C\setminus\set{0,1})
\]
containing configurations $(z_3,\dots,z_n)$
such that the point $(0,1,z_3,\dots,z_n)\in F_n$ belongs to the subspace $e(\pi_r)$. The subspaces $\tilde{e}(\pi_r)$, together
with the point at infinity $\infty$, give a CW structure of the one-point compactification $F_{n-2}(\C\setminus\set{0,1})^+$
of the manifold $F_{n-2}(\C\setminus\set{0,1})$. The reduced cellular cochain complex $\tChb$ of the space $F_{n-2}(\C\setminus\set{0,1})^+$
is by construction isomorphic to the subcomplex of $\Chb$ generated by cells $(\pi_r)$ satisfying the first condition of Definition \ref{defn:tChnlog},
up to a shift in dimension due to the fact that $F_n$ has (real) dimension $2n$, whereas $F_{n-2}(\C\setminus\set{0,1})$
has dimension $2n-4$. The Salvetti complex $\tilde\Sal_n$ is the Poincar\'{e} dual of the cell decomposition of $F_{n-2}(\C\setminus\set{0,1})^+$,
and its cellular chain complex is also isomorphic to $\tChb$.

We can now restrict the covering $p\colon\Fnlog\to F_n$ first to a connected covering $p\colon\tFnlog\to F_{n-2}(\C\setminus\set{0,1})$,
and then to a connected covering $\tilde\Sal^{\log}_n\to\tilde\Sal_n$. Note that $\tFnlog$ is only one connected component of
$p^{-1}\pa{F_{n-2}(\C\setminus\set{0,1})}\subset\Fnlog$: there is indeed one connected component for any fixed value of $w_{12}\in2\pi\sqrt{-1}\Z$.

We pull back the cell structure on $\tilde\Sal_n$ along $p$ to a cell structure on $\tilde\Sal^{\log}_n$; thus 
the chain complex associated with $\tilde\Sal^{\log}_n$ is precisely $\tChlogb$.
\end{proof}

We define filtrations on the chain complexes that we have introduced.

\begin{defn}
\label{defn:filtrationsCh}
For each generator $(\pi_r)_{1\leq r\leq n-k}$ of $\Chb$ there is an index $l$ such that $1\in\pi_l$: we denote $\i(\pi_r)=l$.

We filter $\Chb$ in the following way: a generator $(\pi_r)_{1\leq r\leq n-k}$ in some degree $k$ has \emph{height} $p$, with $0\leq p\leq n-1$,
if there are exactly $p$ indices $i\in\pi_{\i(\pi_r)}$ such that $i\prec 1$. Note that by Definition \ref{defn:Ch}
the height can only decrease along boundaries in $\Chb$.

In the same way we can filter the chain complex $\Chlogb$: a generator $(\pi_r,W_{ij})$ has the same height as the corresponding
generator $(\pi_r)$ of $\Chb$. Note that we obtain a $P_n^{\ab}$-invariant filtration on $\Chlogb$: in other words
$\Chlogb$ becomes a filtered chain complex of $R(n)$-modules.
\end{defn}

The chain complex $\tChlogb$ has a natural action of the group
$H_1(F_{n-2}(\C\setminus\set{0,1}))\simeq\Z^{{\binom{n}{2}}-1}$; as we have already seen,
the group $H_1(F_{n-2}(\C\setminus\set{0,1}))$ can be identified with the
kernel of the map $\psi_{12}\colon P^{\ab}_n\to \Z$, and is generated by elements
$\tilde t_{ij}$ for $1\leq i<j\leq n$ with $(i,j)\neq (1,2)$. Hence $\tChlogb$ can be seen as
a chain complex of free $\tR(n)-$modules.

\begin{defn}
 \label{defn:Omega&filtration}
We consider $\tChlogb$ as a chain complex of free $\tR(n)-$modules and call $\Omega$ the basis containing
those elements $(\pi_r,0)\in\Chlogb$ that lie in $\tChlogb$.

The chain complex $\tChlogb$ inherits a filtration from
$\Chlogb$, with heights $p$ ranging from $0$ to $n-2$:  this is a filtration
in $\tR(n)-$modules.

We call $\fF_p\tChlogb\subset \tChlogb$ the subcomplex generated by cells of height $\leq p$,
and $\fF_p/\fF_{p-1}\tChlogb$ the $p-$th filtration stratum.
\end{defn}

Note that $\Omega$ is a filtered basis for $\tChlogb$.

\section{Morse flows}
In this section we simplify the complex $\tChlogb$ to a chain complex with fewer generators:
we use Forman's discrete Morse theory, which was first introduced in \cite{Forman};
see \cite{Knudson} or \cite{Kozlov} for an introduction to discrete Morse theory.
The Morse complex that we present has already appeared in a similar way in \cite{Djawadi} and \cite{Lofano}.
\begin{defn}
\label{defn:Morseflow}
 Recall from Definition \ref{defn:Omega&filtration} that $\Omega$ is a basis for $\tChlogb$ as
 a chain complex of finitely generated,
 free $\tR(n)$-modules. 
 For a cell $\fe=(\pi_r,0)\in\Omega$, the index $\i(\fe)$ was introduced in Definition \ref{defn:filtrationsCh}.
 We define a matching $\M$ on $\Omega$:
 \begin{itemize}
  \item a cell $\fe=(\pi_r,0)$ is critical if $\iota(\fe)=1$ (i.e. $1\in\pi_1$), and if $1$ is the last element
  of $\pi_1$ according to $\prec$ (i.e. $i\prec 1$ for all $i\in\pi_1$ with $i\neq 1$);
  \item a cell $\fe=(\pi_r,0)$ is collapsible if $1$ is
  not the last element of $\pi_{\i(e)}$. In this case the redundant partner
  of $\fe$ is $\fe'=(\pi'_r,0)$, where $(\pi'_r)$ is obtained from $(\pi_r)$ by splitting
  $\pi_{\i(\fe)}$ into $\pi'_{\i(\fe)}=\set{i\in\pi_l\,|\,1\prec i}$ and
  $\pi'_{\i(\fe)+1}=\set{i\in\pi_l\,|\, i\preceq 1}$,
  as in Definition \ref{defn:Chlog} with $l=\i(\fe)$, and $\prec$ is restricted to the two pieces.
  Informally, we push all elements $i$ lying below $1$ to the left.
  Note that $\i(\fe')=\i(\fe)+1\geq 2$. We write $\fe'\nearrow \fe$, meaning that the couple $(\fe',\fe)$ is in $\M$.
  \item a cell $\fe=(\pi_r,0)$ is redundant if $\i(\fe)\geq 2$ and $1$ is
  the last element of $\pi_{\i(\fe)}$ according to $\prec$. In this case the collapsible partner
  of $\fe$ is $\fe'=(\pi'_r,0)$, where $(\pi'_r)$ is obtained from $(\pi_r)$ by concatenating
  $\pi_{\i(\fe)}$ and $\pi_{\i(\fe)-1}$ into $\pi'_{\i(\fe)-1}$: on the new set $\pi'_{\i(\fe)-1}$ the order $\prec$
  is defined by extending
  $\prec$ on $\pi_{\i(\fe)}$ and $\pi_{\i(\fe)-1}$ with the rule $i\prec j$ for all $i\in\pi_{\i(\fe)}$ and $j\in\pi_{\i(\fe)-1}$.
  In particular $1\prec j$ for all $j\in\pi_{\i(\fe)-1}$. Informally, we push the column on left of 1 underneath 1.
  We write $\fe\nearrow \fe'$.
 \end{itemize}
\end{defn}
 By Definition \ref{defn:Chlog} if two cells $\fe\nearrow \fe'$ are matched, then $[\partial \fe'\colon \fe]$
 is invertible in $\tR(n)$.
 
 To check that $\M$ is acyclic, note first that $\M$ is compatible with the filtration of the chain complex
 $\tChlogb$, hence it suffices to check that $\M$ is acyclic on each filtration stratum $\fF_p/\fF_{p-1}\tChlogb$.
 
 Let $\fe=(\pi_r,0)\nearrow \fe'=(\pi'_r,0)\searrow \fe''=(\pi''_r,0)$ be an alternating path of three distinct cells of degrees $k,k+1,k$,
 all having the same height $p$.
 This means that the redundant cell $\fe$ is matched with the collapsible cell $\fe'$, and that
 $[\partial \fe'\colon \fe'']\neq 0$. Suppose also that $\fe''$ is redundant.
 
 Then both $\fe$ and $\fe''$ are obtained from $\fe'$ by splitting precisely the piece $\pi'_{\i(\fe')}$ as in Definition \ref{defn:Chlog}:
 indeed $1$ is not the last element in $\pi'_{\i(\fe')}$, but is the last element of both $\pi_{\i(\fe)}$ and $\pi''_{\i(\fe'')}$.
 
 Moreover there are exactly two ways to split $\pi'_{\i(\fe')}$ in two pieces, so that the following conditions hold:
 \begin{itemize}
  \item $1$ becomes the last element of its piece;
  \item the height $p$ doesn't decrease, i.e. all elements preceding $1$ in $\pi'_{\i(\fe')}$ still belong to the same piece
  as $1$ and precede $1$. 
 \end{itemize}
The two pieces must be, in some order, $\set{i\in\pi'_{\i(\fe')} \,|\, i\preceq 1}$ and
 $\set{i\in\pi'_{\i(\fe')} \,|\, 1\prec i}$, and we can only choose which piece is split to the left and which to the right.
 
 If $\set{i\in\pi'_{\i(\fe')} \,|\, 1\prec i}$ is split to the left, then we get the redundant partner of $\fe'$, that is, $\fe$;
 in the other case we must get $\fe''$.
 
 We conclude that $\i(\fe)=\i(\fe')+1$, and $\i(\fe'')=\i(\fe')$; in particular $\i(\fe'')>\i(\fe)$. This shows
 that the matching is acyclic on each stratum $p$, because the index $\iota$ strictly increases along alternating paths.
 
 \begin{defn}
  \label{defn:MtChlog}
  We call $\MtChlogb$ the Morse complex associated with the acyclic matching $\M$: it is a chain complex of
  finitely generated, free $\tR(n)$-modules, with basis $\Omega^{\M}$ given by $\M$-critical cells
  in $\Omega$. The chain complex $\M\tChlogb$ is also a filtered chain complex of $\tR(n)$-modules:
  the subcomplex $\fF_p\M\tChlogb$ is generated
  by $\M$-critical cells of height $\leq p$, and the $p$-th filtration stratum is denoted by $\fF_p/\fF_{p-1}\MtChlogb$.
 \end{defn}

 We conclude this section by analysing more carefully the structure of the filtration strata.
 \begin{defn}
 \label{defn:ChbS}
  Let $S$ be a subset of $\set{2,\dots,n}$ containing $2$. We denote by $R(S)$ the ring $\Z[t_{ij}^{\pm 1}]_{i,j\in S,i<j}$.
  This is a domain and is naturally contained in $\tR(n)$; its quotient field is denoted by $\F(S)$,
  and there is an inclusion $\F(S)\subset\tF(n)$. In the particular case $S=\set{2}$ we have $R(S)=\Z$.
  
  Let $\Chb^S$ and $\Ch_{\bullet}^{\log,S}$ be defined in analogy with Definitions \ref{defn:Ch} and \ref{defn:Chlog} but using,
  instead of the set of indices $\set{1,\dots, n}$, its subset $S$. In particular generators of $\Chb^S$ are given
  by ordered partitions $(\pi_r)_{1\leq r\leq\abs{S}-k}$ of $S$; generators of $\Ch_{\bullet}^{\log,S}$
  are given by an ordered partition of $S$ together with a choice of integers $(W_{ij})$ for all $i<j$ with $i,j\in S$.
  
  Note that $\Ch_{\bullet}^{\log,S}$ is a chain
  complex of finitely generated, free $R(S)$-modules, supported in degrees ranging from $0$ to $|S|-1$.
  In the particular case $S=\set{2}$ we have that $\Ch_{\bullet}^{\log,S}$ consists of a copy of $\Z$ in degree 0.
 \end{defn}

 \begin{lem}
  \label{lem:structurefiltrationstrata}
 Let $0\leq p\leq n-2$; then there is an isomorphism of chain complexes of $\tR(n)$-modules
 \[
  \fF_p/\fF_{p-1}\MtChlogb\;\cong\; \bigoplus_{S}\pa{\tR(n)^{p!}\otimes_{R(S)}\Ch_{\bullet}^{\log,S}},
 \]
where the sum is taken over all sets $S\subset\set{2,\dots,n}$ with $|S|=n-p-1$ and $2\in S$.
This isomorphism shifts degrees by $-p$.
 \end{lem}
\begin{proof}
Recall that the differential in the chain complex $\MtChlogb$ is defined as follows: for two $\M$-critical
cells $\fe=(\pi_r,0)$ and $\fe'=(\pi'_r,0)$ in $\Omega^{\M}$ the boundary index $[\partial \fe'\colon \fe]$ is the sum of the weights
of all alternating paths from $\fe'$ to $\fe$.

If $\fe$ and $\fe'$ have the same height $p$, then
an alternating path $\fe'=\fe'_0\searrow \fe_0\nearrow \fe'_1\searrow\dots \searrow \fe_l=\fe$ must contain
only cells of height $p$. Since $\fe'_0$ is critical, $1$ is the last element of $\pi'_1$, and splitting
in two pieces $\pi'_1$ would let the height $p$ of $\fe'_0$ decrease to a smaller height in $\fe_0$:
hence $\fe_0$ is obtained from $\fe'$ by
splitting some other piece $\pi'_l$ with $l\geq 2$, and therefore $\fe_0$ is already critical, hence $\fe_0=\fe$.

Thus the differential in the chain complex $\fF_p/\fF_{p-1}\MtChlogb$ is isomorphic to the differential
obtained from Definition \ref{defn:Chlog} by allowing only a splitting in two pieces of some piece of the partition
$\pi_l$ with $l\geq 2$.

In particular we can split our chain complex  $\fF_p/\fF_{p-1}\MtChlogb$
into many subcomplexes according to \emph{which}
$p$ elements, all different from $2$, appear in $\pi_1$ and \emph{in which order} $\prec$, provided that $1$ is the last element of
$\pi_1$.

To determine one of these subcomplexes we can equivalently choose a set
$S\subset\set{2,\dots,n}$ of $n-p-1$ elements, with $2\in S$,
and declare that the other $p+1$ elements $i\in\set{1,\dots,n}$, including 1, are the elements of $\pi_1$.
Moreover there are exactly $p!$ ways to order these $p+1$ elements inside $\pi_1$, if we require $1$ to be the last in the order:
each of these possible choices of $\prec$ on $\pi_1$ gives rise to a different subcomplex.

Finally we note that each of these subcomplexes is isomorphic to the chain complex $\tR(n)\otimes_{R(S)}\Ch_{\bullet}^{\log,S}$, where the
isomorphism is given by mapping the $\M$-critical cell $(\pi_r,0)_{1\leq r\leq n-k}$ to the cell
$1\otimes (\pi_r,0)_{2\leq r\leq n-k}$: this map has degree $-p$.
\end{proof}

\section{The spectral sequence with coefficients in \texorpdfstring{$\tF(n)$}{tK(n)}}
\label{sec:tFspectral}
In this section we prove that $H_{n-2}(\cPn)\neq 0$. More precisely we prove the following theorem.
\begin{thm}
\label{thm:HtensortF}
For $n\geq 2$ the graded $\tF(n)$-vector space
\[
 \tF(n)\otimes_{\tR(n)}H_*(\cPn)
\]
has dimension $(n-2)!$ in degree $n-2$ and vanishes in all other degrees.
\end{thm}
This means, in particular, that $H_{n-2}(\cPn)$ contains an embedded copy of $\tR(n)^{(n-2)!}$,
which for $n\geq 3$ is a free abelian group of infinite rank.

The following is an immediate consequence of Theorem \ref{thm:HtensortF}.
\begin{cor}
 \label{cor:cdcPn}
 For $n\geq 2$ the cohomological dimension of $\cPn$ is $n-2$.
\end{cor}
\begin{proof}
 We have $\cd(\cPn)\geq n-2$ because $H_{n-2}(\cPn)\neq 0$. Moreover, as already seen in the
 proof of Lemma \ref{lem:tChnlog=ChtFnlog}, the space $\tFnlog$ deformation retracts onto
 the space $\tilde\Sal^{\log}_n$, which is a cell complex of dimension $n-2$; hence $\cd(\cPn)\leq n-2$.
\end{proof}

\begin{proof}[Proof of Theorem \ref{thm:HtensortF}]
We consider the filtered chain complex $\MtChlogb$. Since localisation
is exact we can compute $ H_*(\cPn)\otimes_{\tR(n)}\tF(n)$ as the homology of the chain complex
$\tF(n)\otimes_{\tR(n)}\MtChlogb$, which is a filtered chain complex of $\tF(n)$-vector spaces.

The first page of the associated Leray spectral sequence is
\[
E^1_{p,q}=H_{p+q}\pa{\fF_p/\fF_{p-1}\pa{\tF(n)\otimes_{\tR(n)}\MtChlogb}},
\]
and our aim is to show that the latter groups are all trivial, except for $p=n-2$ and $q=0$,
where we have
\[
H_{n-2}\pa{\fF_{n-2}/\fF_{n-3}\pa{\tF(n)\otimes_{\tR(n)}\MtChlogb}}\simeq\tF(n)^{(n-2)!}.
\]
Once this statement is proved, Theorem \ref{thm:HtensortF} follows immediately because the spectral
sequence collapses on its first page.

By Lemma \ref{lem:structurefiltrationstrata} the chain complex $\fF_{n-2}/\fF_{n-3}\pa{\MtChlogb}$
is isomorphic to the chain complex $\tR(n)^{(n-2)!}\otimes_{R_{\set{2}}}\Chb^{\log,\set{2}}$. Since
the ring $R\pa{{\set{2}}}$ is just $\Z$, and since the chain complex $\Chb^{\log,\set{2}}$ is just a copy
of $\Z$ in degree $0$, we have that the filtration stratum
$\fF_{n-2}/\fF_{n-3}\pa{\MtChlogb}$ is concentrated in degree $n-2$
and its homology is $\tR(n)^{(n-2)!}$, also concentrated in degree $n-2$.

Tensoring with $\tF(n)$ we have
that $E_{n-2,0}\simeq\tF(n)^{(n-2)!}$, and $E_{n-2,q}=0$ for all $q\neq 0$.

We want now to show that the chain complex $\fF_p/\fF_{p-1}\pa{\tF(n)\otimes_{\tR(n)}\MtChlogb}$
is acyclic for all $0\leq p\leq n-3$. By Lemma \ref{lem:structurefiltrationstrata} it suffices to prove that, for any set
$S\subset\set{2,\dots,n}$ containing $2$, the chain complex
\[
 \tF(n)\otimes_{\tR(n)}\tR(n)\otimes_{R(S)}\Ch_S
\]
is acyclic. We note that $\tF(n)$ contains $\F(S)$, so we can equally consider
\[
 \tF(n)\otimes_{\F(S)}\F(S)\otimes_{R(S)}\Ch_S
\]
and the latter is acyclic because $\F(S)\otimes_{R(S)}\Ch_S$ is acyclic by Lemma \ref{lem:H*istauntorsion},
and extending the field $\F(S)\subset\tF(n)$ is exact.
\end{proof}
We note that it was not necessary to localise $\tR(n)$ with respect to all non-zero elements,
i.e. passing from $\tR(n)$ to its quotient field $\tF(n)$.

\begin{defn}
\label{defn:torS}
 Let $S$ be a finite subset of $\set{2,\dots,n}$ containing $2$. We call
\[
 \tor_S=\left[\pa{\prod_{i,j\in S;\,i<j}t_{ij}}-1\right]\in \tR(n)\subset R(n).
\]
Define also
\[
 \tor_n=\prod_{S}\tor_S\in \tR(n)\subset R(n)
\]
where the product is extended over all subsets $S\subset\set{2,\dots,n}$ containing $2$.
\end{defn}
 
Then the same argument of the proof of Lemma \ref{lem:H*istauntorsion}
tells us that, for all subsets $2\in S\subset\set{2,\dots, n}$ with $S\neq \set{2}$, we have
\[
\tR(n)\left[{\tor_n}^{-1}\right]\otimes_{R(S)} H_*\pa{\Chb^{\log,S}}=0.
\]

Therefore we can
repeat the proof of Theorem \ref{thm:HtensortF} to show that
\[
 \tR(n)\left[\tor_n^{-1}\right]\otimes_{\tR(n)}H_*(\cPn)
\]

is concentrated in degree $n-2$, where it is equal to
$\tR(n)\left[\tor_n^{-1}\right]^{(n-2)!}$.

\section{Homology in lower degrees}
In this section we prove non-triviality of $H_*(\cPn)$ in all degrees $*\leq n-2$. More precisely,
we prove the following theorem.
\begin{thm}
 \label{thm:middle}
For all $1\leq *\leq n-2$ the group $H_*(\cPn)$ contains a free abelian group of
infinite rank.
\end{thm}

\begin{proof}
 By Theorem \ref{thm:HtensortF} we know that $H_{n-2}(\cPn)$ contains a free abelian group of infinite rank.
 In the following we fix $3\leq k\leq n-1$ and prove that $H_{k-2}(\cPn)$ has the same property.
 
 Consider the map $\psi^n_k\colon F_n\to F_k$ that forgets the last $n-k$ points of a configuration (compare with
 the maps $\psi_{ij}$ from Definition \ref{defn:Pnab}):
 \[
  \psi^n_k(z_1,\dots,z_n)=(z_1,\dots,z_k)\in F_k.
 \]
 The map $\psi^n_k$ is a fibration (see \cite{FadellNeuwirth}) and there is a section $\sigma^k_n\colon F_k\to F_n$ given by
 adjoining $n-k$ points \emph{far on the right}: formally we set $M(z_1,\dots,z_k)=\max_{i=1}^k\abs{z_i}$ and then we define
 \[
  \sigma^k_n(z_1,\dots,z_k)=(z_1,\dots,z_k,M+1,\dots,M+n-k)\in F_n.
 \]
 We have induced maps on fundamental groups $\psi^n_k\colon P_n\to P_k$ and $\sigma^k_n\colon P_k\to P_n$; the composition
 $\psi^n_k\circ\sigma^k_n\colon P_k\to P_k$ is the identity of $P_k$.
 
 The maps $\psi^n_k$ and $\sigma^k_n$ restrict to maps between commutator subgroups; in particular the composition
 $\psi^n_k\circ\sigma^k_n\colon [P_k,P_k]\to [P_k,P_k]$ is the identity of $[P_k,P_k]$.
 
 This implies that the induced map in homology
 \[
  \pa{\sigma^k_n}_*\colon H_{k-2}([P_k,P_k])\to H_{k-2}(\cPn)
 \]
 is injective, and again by Theorem \ref{thm:HtensortF} we know that $\colon H_{k-2}([P_k,P_k])$ contains a free
 abelian group of infinite rank.
 \end{proof}

 \section{Future directions}
Computing the homology of $\cPn$ as a $R(n)$-module seems a difficult
task, in particular because $R(n)$ is not a principal ideal domain and we lack 
a good classification of finitely generated modules over $R(n)$.
We only observe that $H_*(\cPn)$ is finitely generated over $R(n)$:
indeed the chain complex $\Chlogb$ is finitely generated over $R(n)$, and $R(n)$
is a noetherian ring. 

Computing $H_*(\cPn)$ directly as an abelian group seems not to be easy either. In Theorems \ref{thm:HtensortF}
and \ref{thm:middle} we have proved that $H_k(\cPn)$ contains a free abelian group of infinite rank for $1\leq k\leq n-2$;
we conjecture that $H_*(\cPn)$ is indeed a free abelian group,
and in particular is torsion-free.
Our conjecture is related to a conjecture by Denham \cite{Denham02} on the structure of the homology
of the Milnor fibre of a complexified real arrangement; this conjecture was investigated also by
Settepanella \cite{Settepanella09}. Note that for $n=3$ our conjecture holds, as $\cPn$ is a free group.

Finally, it would also be interesting to study $H_*(\cPn)$ as a representation.
Denote by $B_n$ Artin's braid group on $n$ strands \cite{Artin}, and by
$\mathfrak{S}_n$ the $n$-th symmetric group. There is a short exact sequence
\[
1\to P_n\to B_n\to\mathfrak{S}_n\to 1.
\]
In particular $P_n$ is a normal subgroup of $B_n$; since $\cPn$
is a characteristic subgroup of $P_n$, $\cPn$ is also normal in $B_n$ and we have a short exact sequence
\[
 1\to P_n^{\ab}\cong\Z^{\binom{n}{2}} \to B_n/\cPn\to \mathfrak{S}_n\to 1.
\]
{It would be interesting to understand $H_*(\cPn)$ as a representation of \unskip\parfillskip 0pt \par}
\noindent $B_n/\cPn$.

\section*{Acknowledgments}
The author would like to thank Carl-Friedrich B\"{o}digheimer, Filippo Callegaro, Florian Kranhold, Mark Grant, Davide Lofano,
Martin Palmer, Giovanni Paolini, Oscar Randal-Williams, David Recio-Mitter and Mario Salvetti for useful discussions
and precious comments during the preparation of this paper.

\bibliography{bibliography.bib}{}
\bibliographystyle{plain}
\end{document}